\documentclass[12pt]{amsart}
\usepackage{amscd,amssymb,times,fullpage}
\newtheorem{theorem}{Theorem}

\newtheorem{corollary}[theorem]{Corollary}
\theoremstyle{definition}
\newtheorem{definition}[theorem]{Definition}

\newtheorem{example}[theorem]{Example}

\theoremstyle{remark}
\newtheorem{remark}[theorem]{Remark}

\catcode`\@=\active 

\begin{document}

\title{Moser stability for locally conformally symplectic structures}

\author{G.~Bande}
\address{Dipartimento di Matematica e Informatica, Universit\`a degli Studi di Cagliari, Via Ospedale
72, 09124 Cagliari, Italy}
\email{gbande{\char'100}unica.it}
\author{D.~Kotschick}
\address{Mathematisches Institut, Ludwig-Maximilians-Universit\"at M\"unchen,
Theresienstr.~39, 80333 M\"unchen, Germany}
\email{dieter{\char'100}member.ams.org}


\thanks{This work was carried out while the second author was a Visiting
Professor at the Universit\`a degli Studi di Cagliari, supported
by the Regione Autonoma della Sardegna}
\date{\today;  \copyright{\ G.~Bande and D.~Kotschick 2008}}

\begin{abstract}
We formulate and prove the analogue of Moser's stability theorem \cite{Moser} for locally conformally symplectic structures.
As special cases we recover some results previously proved by Banyaga \cite{banya}.
\end{abstract}

\maketitle

\section{Introduction}

In this paper we prove a version of the Moser stability theorem for locally conformally symplectic structures. These
structures were introduced by Lee \cite{lee} and Vaisman \cite{V}, and have been studied extensively by Vaisman,
Banyaga and many others. We refer the reader to \cite{banya,banya2,DO,HR,lee,V} and to the references given there
for a more thorough discussion.

A locally conformally symplectic or lcs form on a manifold $M$ is a non-degenerate $2$-form $\omega$ which is locally
conformal to a symplectic form. More formally:
\begin{definition}
A non-degenerate $2$-form $\omega$ on a manifold $M$ is said to be
locally conformally symplectic if there exists an open
covering $\{ U_i\}$ of $M$ and a smooth positive function $f_i$ on
each $U_i$ such that $f_i \omega\vert_ {U_i}$ is symplectic on $U_i$.
\end{definition}
It is straightforward to see, and was first observed by Lee \cite{lee}, that this is equivalent to the existence of a
closed $1$-form $\theta$ such that
\begin{equation}\label{eqLee}
d\omega= \theta \wedge \omega .
\end{equation}
We assume throughout that the dimension of $M$ is at least $4$.
Then the $1$-form $\theta$, called the Lee form of $\omega$,
is uniquely determined by $\omega$ because the wedge product with a non-degenerate $2$-form is
injective on $1$-forms. When $\theta$ vanishes identically,
the form $\omega$ is symplectic.

Two lcs forms $\omega$ and $\omega'$ are said to be (conformally) equivalent if there
exists some positive function $f$ such that $\omega = f \omega'$. A locally conformally symplectic \textit{structure} is an equivalence class
of lcs forms for this relation. Note that the de Rham
cohomology class of the Lee form is an invariant of the lcs structure because a conformal rescaling
of $\omega$ changes $\theta$ by the addition of an exact form.

If an lcs structure contains a symplectic representative, then the
structure is \textit{globally conformally symplectic}. This is the case if and only if the Lee form is exact.

The purpose of this paper is to give a necessary and sufficient condition for the existence of an isotopy making a smooth family of lcs structures
constant. This result, Theorem~\ref{t:Moser} proved in Section~\ref{secstab}, is the lcs analogue of Moser's theorem \cite{Moser} for
symplectic forms. However, for the conformally invariant notion of lcs forms, the appropriate formulation is not for an isotopy of forms, but for
an isotopy of their conformal equivalence classes. We will show that Theorem~\ref{t:Moser} implies as special cases some results
of Banyaga \cite{banya} giving sufficient conditions for the existence of an isotopy of certain families of lcs structures. Although we prove
a more general result than Banyaga \cite{banya}, our proofs are simpler, even for the special cases he considered; compare
Section~\ref{secC} below.

The proof of Theorem~\ref{t:Moser} involves Hodge theory for the Lichnerowicz cohomology, which we review in Section~\ref{secL}.
As a byproduct we give an answer to a question raised by Banyaga in \cite{banya2}.

To end this introduction, let us remark on the notation we use. There are a number of conventions, especially concerning signs, involved
in the definition of the Lee form and in the discussion of Lichnerowicz cohomology. We use \eqref{eqLee} as our definition of the Lee form,
and we have chosen the sign conventions for the Lichnerowicz cohomology in such a way that they fit conveniently with \eqref{eqLee}.

\section{Lichnerowicz cohomology}\label{secL}

An important tool in the study of lcs structures is the $d_\theta$-cohomology introduced in \cite{Lichne}, compare also \cite{banya2,DO,HR}.

Let $M$ be a smooth manifold, and $\theta$ a closed $1$-form on $M$. One defines a first order differential operator $d_{\theta}$ as follows:
\begin{eqnarray}
d_{\theta} \beta= d\beta - \theta \wedge \beta \, ,
\end{eqnarray}
where $\beta$ is any differential form. It is straightforward to verify that $d_{\theta}$ squares to zero, so that one obtains a modified
de~Rham complex $(\Omega^*(M),d_{\theta})$. Its cohomology vector spaces $H_{\theta}^* (M)$ are
called the $d_{\theta}$-cohomology, or Lichnerowicz cohomology of $M$ with respect to $\theta$. This only
depends on the de Rham cohomology class of $\theta$, for if $\theta'=\theta + d\ln f$ for some positive function $f$, then
he formula
\begin{equation}\label{chain}
f d_{\theta}\beta = d_{\theta+d\ln f}(f\beta)
\end{equation}
shows that multiplication by $f$ is a chain map between $(\Omega^*(M),d_{\theta})$ and $(\Omega^*(M),d_{\theta'})$ inducing an isomorphism in
cohomology.

In the case when $\theta$ is the Lee form of an lcs form $\omega$, equation~\eqref{eqLee} shows that $\omega$ is $d_{\theta}$-closed, and
so defines a class in $H_{\theta}^2 (M)$. If we consider the lcs structure defined by $\omega$, and $\omega' = f \omega$, then the Lee form of
$\omega'$ is just $\theta'=\theta + d\ln f$, and the class $[\omega]\in H_{\theta}^2 (M)$ is mapped to $[\omega']\in H_{\theta'}^2 (M)$ by the
above isomorphism.

The Lichnerowicz  cohomology shares many properties with the ordinary de Rham cohomology, see for example \cite{Lichne, HR}.
For our purposes it is useful that Hodge theory applies to the $d_{\theta}$-cohomology.

Let us assume that $M$ is closed and oriented. Then the modified de~Rham complex $(\Omega^*(M),d_{\theta})$ is an elliptic
complex. In particular, its cohomology is finite-dimensional. If we equip $M$ with an arbitrary Riemannian metric $g$, then we can
define an operator $d_{\theta}^*$ as the formal $L^2$-adjoint of $d_{\theta}$ with respect to $g$. Further,
$\Delta_{\theta}= d_{\theta}d_{\theta}^*+d_{\theta}^*d_{\theta}$ is the corresponding Laplacian. These operators are lower-order perturbations of the
corresponding operators in the usual Hodge--de~Rham theory (corresponding to $\theta =0$), and therefore have much the same analytic properties.
For example, the usual proof of the Hodge decomposition theorem, see e.g. \cite{Warner}, goes through, and one obtains
an orthogonal decomposition
$$
\Omega^k(M) = \mathcal{H}^k(M)\oplus d_{\theta}(\Omega^{k-1}(M))\oplus d_{\theta}^*(\Omega^{k+1}(M)) \ ,
$$
where $\mathcal{H}^k(M)$, the space of $\Delta_{\theta}$-harmonic forms, is isomorphic to $H_{\theta}^k (M)$.

Any $d_{\theta}$-exact $k$-form is contained in $d_{\theta}(\Omega^{k-1}(M))$, which is in fact equal to $d_{\theta}d_{\theta}^*(\Omega^{k}(M))$.
Integration by parts shows that
$$
d_{\theta}\colon d_{\theta}^*(\Omega^{k}(M))\longrightarrow \Omega^{k}(M)
$$
is injective, so that every $d_{\theta}$-exact form has a unique primitive in the image of $d_{\theta}^*$.

Finally, note that on a closed oriented manifold the index of the elliptic complex $(\Omega^*(M),d_{\theta})$ is determined, via the
Atiyah--Singer index theorem, by its symbol sequence, which is independent of $\theta$. Therefore, the Euler characteristic of the
Lichnerowicz  cohomology coincides with the usual Euler characteristic.

\begin{example}
Banyaga \cite{banya2} considers the Lichnerowicz cohomology on a certain $4$-manifold $M$ of Euler characteristic
zero. For a particular closed one-form $\theta$ he shows that the dimensions of $H^i_{\theta}(M)$ are at least
one for $i= 1,2,3$, and asks whether the dimensions might be exactly one, see Question~3 on page~5 of \cite{banya2}.

In this case the one-form $\theta$ is not exact, and so a result of \cite{Lichne,HR} shows that $H^i_{\theta}(M)$ vanishes
for $i=0$ and $i=4$. Therefore the vanishing of the Euler characteristic implies that  $H^2_{\theta}(M)$ is at least $2$-dimensional,
giving a negative answer to Banyaga's question.
\end{example}

\section{Moser stability}\label{secstab}

In this section we consider families $\omega_t$ of locally conformally symplectic forms depending smoothly on a parameter
$t\in [0,1]$. The uniqueness of the Lee form $\theta_t$ implies that this depends smoothly on $t$ as well.

Recall that Moser's stability theorem \cite{Moser} says that if the $\omega_t$ are actually symplectic, then they are isotopic
as forms if and only if their de~Rham cohomology class is independent of $t$. Now one expects to transpose this statement to the lcs
category by replacing the de~Rham cohomology by the Lichnerowicz cohomology. This is slightly complicated for two reasons. First of
all, the Lichnerowicz cohomology depends on $\theta_t$, which is not necessarily fixed, but varies with $t$. Thus the cohomology one
needs to consider also varies with $t$. Second of all, we would like to have a conformally invariant statement, which holds for
lcs structures rather than forms. This explains why the following theorem looks more complicated than Moser's.

\begin{theorem}\label{t:Moser}
Let $\omega_t$ be a family of locally conformally symplectic forms on a closed manifold $M$, depending smoothly on $t\in [0,1]$.
Denote by $\theta_t$ the Lee form of $\omega_t$.

There exists an isotopy $\phi_t$ with $\phi_t^*\omega_t$ conformally equivalent to $\omega_0$ for all $t$ if and only if there are 
positive smooth functions $f_t$ on $M$, varying smoothly with $t$, such that the time derivative $\frac{d}{dt}(f_t\omega_t)$ of the 
conformally rescaled family $f_t\omega_t$ is $d_{\theta_t'}$-exact for every $t$, where $\theta_t'=\theta_t+d\ln f_t$ is the Lee form 
of $f_t\omega_t$.
\end{theorem}
\begin{proof}
First suppose that there is an isotopy $\phi_t$ so that $\phi_t^*\omega_t$ is conformally equivalent to $\omega_0$ for all $t$. After
rescaling the $\omega_t$ suitably, we may assume $\phi_t^*\omega_t=\omega_0$. Let $X_t$ be the time-dependent vector field obtained
by differentiating $\phi_t$.
By the  Cartan formula we have:
\begin{alignat*}{1}
0 &=\frac{d}{dt}(\phi_t^*\omega_t) = \phi_t^*(\dot\omega_t+L_{X_t}\omega_t)\\
&=\phi_t^*(\dot\omega_t+di_{X_t}\omega_t+i_{X_t}d\omega_t)\\
&=\phi_t^*(\dot\omega_t+di_{X_t}\omega_t+i_{X_t}(\omega_t\wedge\theta_t))\\
&=\phi_t^*(\dot\omega_t+di_{X_t}\omega_t-\theta_t\wedge (i_{X_t}\omega_t)+\theta_t(X_t)\omega_t) \ .
\end{alignat*}
We conclude that
\begin{equation}\label{eq1}
\dot\omega_t =   -d_{\theta_t}( i_{X_t}\omega_t) -\theta_t(X_t)\omega_t \ .
\end{equation}
Now consider the  family $f_t\omega_t$ with
$$
f_t = \exp \left( \int_0^t \theta_s(X_s) ds \right) \ .
$$
This satisfies
\begin{alignat*}{1}
\frac{d}{dt} (f_t\omega_t) &= \dot f_t \omega_t + f_t \dot\omega_t\\
&=f_t\theta_t(X_t)\omega_t-f_t(d_{\theta_t}( i_{X_t}\omega_t) +\theta_t(X_t)\omega_t )\\
&=-f_t d_{\theta_t}( i_{X_t}\omega_t)\\
& = -d_{\theta_t+d\ln f_t}(f_t i_{X_t}\omega_t) \ ,
\end{alignat*}
where we first used \eqref{eq1}, and then \eqref{chain}. Thus $\frac{d}{dt} (f_t\omega_t)$ is indeed $d_{\theta_t'}$-exact,
where $\theta_t'=\theta_t+d\ln f_t$ is the Lee form of $f_t\omega_t$.

For the converse assume that we have  rescaled $\omega_t$ in such a way that $\dot\omega_t$ is $d_{\theta_t}$-exact for every $t$.
As explained in Section~\ref{secL}, by Hodge theory for the Lichnerowicz cohomology, there is a unique $\alpha_t$ in the image of $d_{\theta_t}^*$ with the property
that $d_{\theta_t}\alpha_t=\dot\omega_t$. As $\dot\omega_t$ depends smoothly on $t$, and $\alpha_t$ is specified uniquely for every $t$,
it follows that $\alpha_t$ also depends smoothly on $t$. We define a time-dependent vector field $X_t$ by the requirement that
$$
i_{X_t}\omega_t = -\alpha_t \ .
$$
This exists and is unique because $\omega_t$ is non-degenerate. As both $\omega_t$ and $\alpha_t$ depend smoothly on $t$,
so does $X_t$. Let $\phi_t$ be its flow. Now the same calculation as before yields
\begin{alignat*}{1}
\frac{d}{dt}(\phi_t^*\omega_t) &= \phi_t^*(\dot\omega_t+L_{X_t}\omega_t)\\
&=\phi_t^*(\dot\omega_t+di_{X_t}\omega_t-\theta_t\wedge (i_{X_t}\omega_t)+\theta_t(X_t)\omega_t)\\
&=\phi_t^*(\dot\omega_t-d_{\theta_t}\alpha_t +\theta_t(X_t)\omega_t)\\
&=\phi_t^*(\theta_t(X_t)) \cdot \phi_t^*\omega_t \ .
\end{alignat*}
It follows that
$$
\phi_t^*\omega_t =  \exp \left(\int_0^t \phi_s^*(\theta_s(X_s))ds\right) \cdot \omega_0 \ .
$$
This completes the proof.
\end{proof}
Note that in the second part of the proof we produced an isotopy $\phi_t$ between the conformal equivalence classes of the $\omega_t$.
As the de~Rham cohomology class of the Lee form is conformally invariant, and $\phi_t$ acts trivially on de~Rham cohomology, we
conclude that the de~Rham cohomology class of the Lee forms $\theta_t$ is independent of $t$. We do not need to assume this in the theorem,
as the proof shows that it is implied by the assumption that $\frac{d}{dt}(f_t\omega_t)$ is $d_{\theta_t'}$-exact for all $t$.

If we assume that the Lee forms are independent of $t$, then Theorem~\ref{t:Moser} implies the following statement, first proved by
Banyaga (Theorem~4 of \cite{banya}):
\begin{corollary}
Let $\omega_t$ be a smooth family of lcs forms on a compact manifold $M$ having the same Lee form $\theta$. If $\omega_t - \omega _0$ is
 $d_{\theta}$-exact for all $t$, then there exist a family of functions $f_t$ and an isotopy $\phi_t$ such that $\phi_t^* (\omega_t)= f_t \omega_0$.
\end{corollary}
Another special case of Theorem~\ref{t:Moser} previously proved by Banyaga is the following (Theorem~5 of \cite{banya}):
\begin{corollary}\label{c2}
Let $\omega_t$ a smooth family of lcs forms on a compact manifold $M$ such that the corresponding Lee forms $\theta_t$ have the same de Rham cohomology
class. Suppose there exists a smooth family of $1$-forms $\alpha_t$ such that $\omega_t= d \alpha_t-\theta_t \wedge \alpha_t$. Then there exists an
isotopy $\phi_t$ such that $\phi_t^*\omega_t$ is conformally equivalent to $\omega_0$ for all $t$.
\end{corollary}
\begin{proof}
In this case we have
\begin{equation}\label{eq3}
\dot\omega_t= d \dot\alpha_t-\theta_t \wedge \dot\alpha_t -
\dot\theta_t\wedge\alpha_t = d_{\theta_t}\dot\alpha_t -
\dot\theta_t\wedge\alpha_t \ .
\end{equation}
Because of the second summand on the right-hand side, it is not obvious that Theorem~\ref{t:Moser} applies. However, here one has the additional
assumption that the de~Rham cohomology class of $\theta_t$ is independent of $t$. This means that $\dot\theta_t = dh_t$ for some smoothly varying
family of functions $h_t$. Let
$$
g_t = -\int_0^t h_s ds
$$
and $f_t = \exp g_t$. With these definitions we have
\begin{alignat*}{1}
\frac{d}{dt} (f_t\omega_t) &= \dot f_t \omega_t + f_t \dot\omega_t\\
&=f_t\dot g_t\omega_t+f_t(d_{\theta_t}\dot\alpha_t - \dot\theta_t\wedge\alpha_t)\\
&=f_t (-h_t(d \alpha_t-\theta_t \wedge \alpha_t)+d_{\theta_t}\dot\alpha_t - \dot\theta_t\wedge\alpha_t)\\
& = f_t d_{\theta_t}(-h_t\alpha_t+\dot\alpha_t)\\
&= d_{\theta_t+d\ln f_t}(f_t(-h_t\alpha_t+\dot\alpha_t)) \ ,
\end{alignat*}
where we used first \eqref{eq3}, then $\dot\theta_t=dh_t$, and finally \eqref{chain}.

We are now in a position to appeal to the sufficiency proof for the existence of the isotopy from (the proof of) Theorem~\ref{t:Moser}. However, as
we have an explicit smooth family of $d_{\theta_t+d\ln f_t}$-primitives for $\frac{d}{dt} (f_t\omega_t)$, Hodge theory is not needed.
\end{proof}

\begin{remark}
Instead of deducing Corollary~\ref{c2} from Theorem~\ref{t:Moser}, one can give a quick direct proof as follows.
Given a smooth family of functions $h_t$ such that $\dot\theta_t=d h_t$, one defines a time-dependent vector field $X_t$ by
\[
i_{X_t} \omega_t= -\dot\alpha_t + h_t \alpha_t \, .
\]
Then its flow $\phi_t$ satisfies
$$
\frac{d}{dt} (\phi_t^*\omega_t) =\phi_t^* ((\theta_t (X_t) +h_t) \omega_t)
$$
by the same kind of calculation as above, and so provides the desired isotopy.
\end{remark}

\section{Conclusion}\label{secC}

In this paper we have given a necessary and sufficient condition for the existence of an isotopy $\phi_t$ making a smooth family of
lcs structures constant. The necessity had not appeared in the literature before. Banyaga \cite{banya} had discussed special
cases of the sufficiency. Comparing our arguments with his, the main difference is that we work directly on a closed manifold $M$
and construct the desired isotopy by integrating a time-dependent vector field on $M$, whereas Banyaga \cite{banya}
works on a covering of $M$ which is no longer compact. The non-compactness leads to two complications. One is that completeness
of (time-dependent) vector fields is not always available, and care has to be taken to ensure it in the situation at hand. The other
complication is that Hodge theory is not available, and the only way to find smooth families of primitives for $d_{\theta}$-exact forms
is Grothendieck's theory of nuclear spaces.


\begin{thebibliography}{9}

\bibitem{banya}
A. Banyaga, \textit{Some properties of locally conformal
symplectic structures}, Comment. Math. Helv. \textbf{77} (2002), no. 2, 383--398.

\bibitem{banya2}
A. Banyaga, \textit{Examples of non $d_\omega$-exact locally conformal symplectic forms}, J. Geom. \textbf{87} (2007), 1--13.

\bibitem{DO}
S. Dragomir and L. Ornea, \textsl{Locally conformal K\"ahler
geometry}, Progress in Mathematics \textbf{155}, Birkh\"auser
Boston, Inc., Boston, MA 1998.

\bibitem{Lichne}
F. Gu\'edira and A. Lichnerowicz, \textit{G\'eom\'etrie des alg\`ebres de Lie locales de Kirillov}, J. Math. Pures Appl. \textbf{63} (1984), no. 4, 407--484.

\bibitem{HR}
S. Haller and T. Rybicki, \textit{On the group of diffeomorphisms
preserving a locally conformal symplectic structure}, Ann. Global Anal. Geom. \textbf{17} (1999), no. 5, 475--502.

\bibitem{lee}
H. C. Lee, \textit{A kind of even-dimensional differential
geometry and its application to exterior calculus}, Amer. J.
of Math. \textbf{65} (1943), 433--438.

\bibitem{Moser}
J. Moser, \textit{On the volume elements on a manifold},
Trans. Amer. Math. Soc. \textbf{120}
(1965), 286--294.

\bibitem{V}
I. Vaisman, \textit{Locally conformal symplectic manifolds}, Internat. J. Math. $\&$ Math. Sci. \textbf{8} (1985), no. 3, 521--536.

\bibitem{Warner}
F. W. Warner, \textsl{Foundations of differentiable manifolds and
Lie groups}, Scott, Foresman and Co., Glenview, Ill--London 1971.

\end{thebibliography}
\end{document}